\documentclass[reqno]{amsart}
\usepackage{amsfonts}
\usepackage{amsmath,amsthm}
\usepackage{amsfonts,amssymb}
\usepackage{mathrsfs}
\usepackage{enumerate}
\usepackage[inner=2.5cm,outer=2.5cm]{geometry}
\allowdisplaybreaks

\newtheorem{thm}{Theorem}[section]

\newtheorem{lem}[thm]{Lemma}
\newtheorem{prop}[thm]{Proposition}

\theoremstyle{remark}
\newtheorem{rem}[thm]{Remark}

\numberwithin{equation}{section}

\newcommand{\al}{\alpha}

\def\({\Bigl(}
\def \){ \Bigr)}

 \def\RR{{\mathbb R}}

        \def\supp{\operatorname{supp}}

\begin{document}
\def\RR{\mathbb{R}}
\def\Exp{\text{Exp}}
\def\FF{\mathcal{F}_\al}

\title[] {Imaginary powers of $(k,1)$-generalized harmonic oscillator}

\author[]{Wentao Teng}

\address{School of Science and Technology, Kwansei Gakuin University, Japan.}

\email{ wentaoteng6@sina.com}

\keywords { imaginary powers; $(k,1)$-generalized harmonic
oscillator; H\"ormander type condition.} \subjclass[2000]{ 
43A85, 31B10, 42B20, 53C23.}

\begin{abstract} In this paper we will define and investigate the imaginary powers $\left(-\triangle_{k,1}\right)^{-i\sigma},\sigma\in\mathbb{R}$ of the $(k,1)$-generalized harmonic oscillator $-\triangle_{k,1}=-\left\|x\right\|\triangle_k+\left\|x\right\|$ and prove the $L^p$-boundedness $(1<p<\infty)$ and weak $L^1$-boundedness of such operators. It is a parallel result to the $L^p$-boundedness $(1<p<\infty)$ and weak $L^1$-boundedness of the imaginary powers of the Dunkl harmonic oscillator $-\triangle_k+\left\|x\right\|^2$.
To prove this result, we develop the Calder\'on--Zygmund theory
adapted to the $(k,1)$-generalized setting by constructing the
metric space of homogeneous type corresponding to the
$(k,1)$-generalized setting, and show that
$\left(-\triangle_{k,1}\right)^{-i\sigma}$ are singular integral
operators satisfying the corresponding H\"ormander type condition.

\end{abstract}

\maketitle
\input amssym.def

\section{Introduction}
Dunkl theory is a far-reaching generalization of classical Fourier analysis related to root system initiated by Dunkl \cite{Du1}. In the past twenty years, there have been many studies on Dunkl theory, including maximal functions, Bochner--Riesz means, multipliers, Riesz transforms and Calder\'on--Zygmund theory.
In \cite{NS1}, A. Nowak and K. Stempak studied the Riesz
transforms related to the Dunkl harmonic oscillator $L_k$ and
proved when the finite reflection group $G$ is isomorphic to
$\mathbb{Z}_2^N$ such operators are $L^p$-bounded $(1<p<\infty)$
and weakly $L^1$-bounded. They then continued the study of the
Dunkl harmonic oscillator $L_k$ in \cite{NS2} by considering its
imaginary powers $L_k^{-i\sigma},\;\sigma\in\mathbb{R}$ and proved
the $L^p$-boundedness $(1<p<\infty)$ and weak $L^1$-boundedness of
such operators when $G$ is isomorphic to $\mathbb{Z}_2^N$. It is a
generalization of the result obtained by Stempak and Torrea
\cite{ST} on the imaginary powers of classical harmonic oscillator
corresponding to the case when the multiplicity function
$k\equiv0$. In \cite{A}, B. Amri extended the result in \cite{NS1}
to general finite reflection groups $G$ in arbitrary dimensions.
And the result in \cite{NS2} can also be extended to general
finite reflection groups using the same techniques as in \cite{A}
with slight modifications. The principal tool of the above results
is the adaptation of the classical Calder\'on--Zygmund theory to
Dunkl setting due to Amri and Sifi \cite{AS}. 

The framework of
Dunkl theory is as follows: Given a root system $R$ in the
Euclidean space $\mathbb R^N$, denote by $\sigma_\alpha$ the
reflection in the hyperplane orthogonal to $\alpha$ and $G$ the
finite subgroup of $O(N)$ generated by the reflections
$\sigma_\alpha$ associated to the root system. Define a
multiplicity function $k:R\rightarrow \mathbb C$ such that $k$ is
$G$-invariant, that is, $k\left(\alpha\right)=k\left(\beta\right)$
if $\sigma_\alpha$ and $\sigma_\beta$ are conjugate. The
$Dunkl\;operators\;T_j$, $1\leq j\leq N$, which were introduced in
\cite{Du1}, are defined by the following deformations by
difference operators of directional derivatives $\partial_j$:
\begin{align*}T_j f(x)=\partial_j f(x)+\sum_{\alpha\in
R^+}k(\alpha)\alpha_j\frac{f(x)-f
(\sigma_\alpha(x))}{\left\langle\alpha,\;x\right\rangle},\end{align*}
where $\left\langle\cdot,\cdot\right\rangle$ denotes the standard
Euclidean inner product and $R^+$ is any fixed positive subsystem
of $R$. They commute pairwise and are skew-symmetric with respect
to the $G$-invariant measure $dm_k(x)=h_k(x)dx$, where
$h_k(x)=\prod_{\alpha\in
R}\vert\left\langle\alpha,\;x\right\rangle\vert^{k(\alpha)}$. The
Dunkl harmonic oscillator is defined as
$L_k=-\triangle_k+\left\|x\right\|^2$, where $\triangle_k$ denotes
the Dunkl Laplacian $\triangle_k={\textstyle\sum_{j=1}^N}T_j^2$.
When $k\equiv 0$, the operator $L_k$ recedes to the classical
harmonic oscillator $-\triangle+\left\|x\right\|^2$, where
$\triangle$ stands for the classical Euclidean Laplacian. The eigenfunction of $\triangle_k$ for fixed $y$ is the integral kernel of the generalized Fourier transform called Dunkl transform. It takes the place of the exponential function $e^{-i\left\langle x,y\right\rangle}$ in classical Fourier transform. 

The operators $\partial_j$ and $T_j$ are intertwined by a
Laplace-type operator (see \cite{Du2})
\begin{align}V_kf(x)=\int_{\mathbb{R}^N}f(y)d\mu_x(y)\end{align}
associated to a family of probability measures
$\left\{\mu_x\vert\;x\in\mathbb{R}^N\right\}$ with compact support
(see \cite{R2}), that is, $T_j\circ V_k=V_k\circ \partial_j$.
Specifically, the support of $\mu_x$ is contained in the convex
hull $co(G\cdot x)$, where $G\cdot x=\left\{g\cdot x\vert\;g\in
G\right\}$ is the orbit of $x$. For any Borel set $B$ and any
$r>0$, $g\in G$, the probability measures satisfy
$$\mu_{rx}\left(B\right)=\mu_x\left(r^{-1}B\right),\;\mu_{gx}\left(B\right)=\mu_x\left(g^{-1}B\right).$$
The intertwining operator $V_k$ is one of the most important operators in Dunkl theory.

More recently, S. Ben Sa\"id, T. Kobayashi and B. \O rsted \cite{BSK2} gave a
further far-reaching generalization of Dunkl theory by introducing
a parameter $a>0$ arisen from the “interpolation” of the two
$sl(2,\mathbb R)$ actions on the Weil representation of the
metaplectic group $Mp(N,\mathbb R)$ and the minimal unitary
representation of the conformal group $O(N+1,2)$. They defined the
$a$-deformed Dunkl harmonic oscillator as
$\triangle_{k,a}:=\left\|x\right\|^{2-a}\triangle_k-\left\|x\right\|^a$.
The operator is an essentially self-adjoint operator on
${{L}^{2}}\left( {{\mathbb{R}}^{N}},{\vartheta_{k,a}}\left( x
\right)dx \right)$ with only negative discrete spectrum, where
${\vartheta_{k,a}}\left( x \right)={{\left\| x
\right\|}^{a-2}}h_k(x)$. They then proved the existence of a
$(k,a)$-generalized holomorphic semigroup  ${\mathcal
I}_{k,a}\left(z\right):=\exp \left(\frac
za\triangle_{k,a}\right),\;\Re z\geq0$ with infinitesimal
generator $\frac1a\triangle_{k,a}$. This holomorphic semigroup
recedes to the Hermite semigroup studied by Howe \cite{H} when
$k\equiv0$ and $a=2$; to the Laguerre semigroup studied by
Kobayashi and Mano \cite{KM1,KM2} when $k\equiv0$ and $a=1$; to
the Dunkl Hermite semigroup studied by R\"osler \cite{R1} when
$k\geq 0$, $a=2$ and $z=2t,\;t>0$. The $(k,a)$-generalized Fourier
transform is then defined as the  boundary value $z=\frac{\pi i}2$
of the semigroup, i.e., $F_{k,a}=e^{i\pi(\frac{2\left\langle
k\right\rangle+N+a-2}{2a})}{\mathcal I}_{k,a}\left(\frac{\pi
i}2\right),$ where $\left\langle k\right\rangle:=\sum_{\alpha\in
R^+}k(\alpha)$.

The $(k,a)$-generalized Fourier
transform defined on ${{L}^{2}}\left( {{\mathbb{R}}^{N}},{\vartheta_{k,a}}\left( x
\right)dx \right)$ has its integral representation as (see \cite[(5.8)]{BSK2}) 
$$F_{k,a}f\left(\xi\right)=c_{k,a}\int_{\mathbb{R}^N}f\left(y\right)B_{k,a}\left(\xi,y\right)\vartheta_{k,a}\left(y\right)dy,\;\;\xi\in\mathbb{R}^N,$$
where $c_{k,a}$ is a constant and $B_{k,a}\left(x,y\right)$ is a symmetric kernel.
It recedes to the Dunkl transform when $a=2$ for $f\in{({L}^{1}\cap {L}^{2})}\left( {{\mathbb{R}}^{N}},{\vartheta_{k,a}}\left( x
\right)dx \right)$. 
In \cite[Theorem 5.11]{BSK2}, the authors showed that the integral kernel $B_{k,a}\left(x,y\right)$ satisfies the condition
\begin{align}\label{uniform}\left|B_{k,a}\left(x,y\right)\right|\leq\left|B_{k,a}\left(0,y\right)\right|\leq 1\end{align}
if $a=1$ or $2$ assuming that $2\left\langle k\right\rangle+N+a-3\geq0$. 
In this case one can define the $(k,a)$-generalized translation operator via an integral combining the inversion formula of the $(k,a)$-generalized Fourier transform for $a=\frac2n,\;n\in{\mathbb{N}}$ (see \cite[Theorem 5.3]{BSK2}). For the general case of $2\left\langle k\right\rangle+N+a-3\geq0$, the condition of boundedness \eqref{uniform} is not necessarily true and it remains an open problem whether it holds (see \cite [Section 6]{GI2}). In \cite{CD}, the authors proved such boundedness for $k\equiv0$ and $a=\frac2n,\;n\in{\mathbb{N}}$ only. 

In \cite{BD}, S. Ben Sa\"id and L. Deleaval
studied the particular case when $k>0$ and $a=1$. They defined the generalized translation operator via an integral in this setting and derived a positive radial formula of the $(k,1)$-generalized translation from the product of the integral kernel of the generalized Fourier transform. They then found many
parallel results to Dunkl's analysis (the case when $k\geq 0$ and
$a=2$) from such definition of the generalized translation operator. We continue the study of the $(k,1)$-generalized Fourier analysis by S. Ben Sa\"id and L. Deleaval \cite{BD} in this paper.

We will define and investigate the imaginary powers
$\left(-\triangle_{k,1}\right)^{-i\sigma},\sigma\in\mathbb{R}$ of
the $(k,1)$-generalized harmonic oscillator $-\triangle_{k,1}=-\left\|x\right\|\triangle_k+\left\|x\right\|$ and
prove a parallel result for general finite reflection groups $G$
to that in \cite{NS2}. According to the radial formula of
$(k,1)$-generalized translation operator given in \cite{BD}, we
develop the Calder\'on--Zygmund theory adapted to the
$(k,1)$-generalized setting by constructing the metric space of
homogeneous type corresponding to the $(k,1)$-generalized setting
and giving the corresponding H\"ormander type condition to prove
the $L^p$-boundedness $(1<p<\infty)$ and weak $L^1$-boundedness of
$\left(-\triangle_{k,1}\right)^{-i\sigma}$.

For a general metric space, a well-known definition of
differatiation by Cheeger \cite{Ch} is given via integration on
continuous rectifiable curves. Unfortunately, rectifiable curves
between two distinct points do not necessarily exist (or in other words, the induced length metric could be infinite) with respect
to the metric corresponding to $(k,1)$-generalized analysis and
derivatives on the metric space cannot be defined. We will make use of an estimate of
difference quotient analogue in substitute of estimate of
derivative.
In the $(k,1)$-generalized setting, it is reasonable
to consider the operator $\left\|x\right\|^{}\triangle_k$ as the
$(k,1)$-generalized Laplacian because the distribution kernel of
the $(k,1)$-generalized Fourier transform is the eigenfunction of
the operator $\left\|x\right\|^{}\triangle_k$ (see \cite[Theorem
5.7]{BSK2}). The imaginary powers of $(k,1)$-generalized harmonic oscillators motivates us to develop the Calder\'on--Zygmund theory in $(k,1)$-generalized setting.
 The development of $(k,1)$-generalized Fourier analysis is still at its infancy and there have only been \cite{BD} and \cite{BD2} on this field.

The paper is organized as follows. In Section 2 we recall some
results in $(k,1)$-generalized Fourier analysis and the
translation operator in this setting. In Section 3, we will study
the corresponding metric space of homogeneous type and develop the
Calder\'on--Zygmund theory adapted to the $(k,1)$-generalized
setting. In Section 4, we define and investigate the imaginary
powers $\left(-\triangle_{k,1}\right)^{-i\sigma}$ of the
$(k,1)$-generalized harmonic oscillator and state the main
theorem. In the last section we will show that such operators
satisfy the corresponding H\"ormander type condition given in
Section 3 to prove the main theorem. We assume $k>0$ in this paper
and most of the results will be under the condition $2\left\langle
k\right\rangle+N-2> 0$. Throughout the paper we denote $C$, $C_1$,
$C_2$  to be constants varying from line to line and $b$, $c$,
$b_1$, $b_2$ to be some positive absolute constants. The root
system we are concerned with is not necessarily crystallographic.

\section{Preliminaries}
The study of Dunkl theory originates from a generalization of
spherical harmonics with the Dunkl weight measure
$dm_k(x)=h_k(x)dx$, where $h_k(x)=\prod_{\alpha\in
R}\vert\left\langle\alpha,\;x\right\rangle\vert^{k(\alpha)}$,
playing the role of Lebesgue measure $dx$ in the classical theory
of spherical harmonics. Let $P_m$ be the space of homogeneous
polynomials on $\mathbb{R}^N$ of degree $m$. The so called Dunkl
Laplacian $\triangle_k$ was constructed in such a way that
$P_m\cap ker\triangle_k$ are orthogonal to each other for
$m=0,1,\cdots$ with respect to Dunkl weight measure $m_k$. It has
the following explicit expression,
$$\triangle_k f\left(x\right)=\triangle_{}f\left(x\right)+2\sum_{\alpha\in R^+}k\left(\alpha\right)\left(\frac{\left\langle\nabla_{}f,\alpha\right\rangle}{\left
\langle\alpha,x\right\rangle}-\frac{f\left(x\right)-f\left(\sigma_\alpha\left(x\right)\right)}{\left\langle\alpha,x\right\rangle^2}\right).$$
Denote $\mathcal H_k^m\left(\mathbb{R}^N\right):=P_m\cap
ker\triangle_k$ to be the space of $h$-harmonic polynomials of
degree $m$. Then the elements in the restriction $\mathcal
H_k^m\left(\mathbb{R}^N\right)\vert_{\mathbb S^{N-1}}$ of
$\mathcal H_k^m\left(\mathbb{R}^N\right)$ to the unit sphere
$\mathbb S^{N-1}$ were called spherical $h$-harmonics. The spaces
$\mathcal H_k^m\left(\mathbb{R}^N\right)\vert_{\mathbb
S^{N-1}},\;m=0,1,\cdots$ are finite dimensional and there is the
spherical harmonics decomposition
\begin{equation}\label{decomposition}L^2\left(\mathbb S^{N-1},h_k\left(x'\right)d\sigma(x')\right)=\sum_{m\in \mathbb N}^\oplus\mathcal H_k^m\left(\mathbb{R}^N\right)\vert_{\mathbb S^{N-1}},\end{equation}
 where $d\sigma$ denotes the spherical measure. For each fixed $m\in\mathbb{N}$, denote by
$d(m)=\dim\left(\mathcal
H_k^m\left(\mathbb{R}^N\right)\vert_{\mathbb S^{N-1}}\right)$. Let
$\left\{Y_i^m:\;i=1,2,\cdots,d(m)\right\}$ be an orthonormal basis
of $\mathcal H_k^m\left(\mathbb{R}^N\right)\vert_{\mathbb
S^{N-1}}$. They are the eigenvectors of the generalized
Laplace--Beltrami operator ${\left.\triangle_k\right|}_{\mathbb
S^{N-1}}$.

From the spherical harmonic decomposition \eqref{decomposition} of
$L^2\left(\mathbb S^{N-1},h_k\left(x'\right)d\sigma(x')\right)$,
there is a unitary isomorphism (see \cite[(3.25)]{BSK2})
$$\sum_{m\in \mathbb N}^\oplus(\mathcal H_k^m\left(\mathbb{R}^N\right){\vert_{\mathbb S^{N-1}})\otimes L^2}\left({\mathbb{R}}_+,r^{\;2\left\langle k\right\rangle+N-2}dr\right)\xrightarrow\sim L^2\left(\mathbb{R}^N,\vartheta_{k,1}\left(x\right)dx\right),$$
where ${\vartheta_{k,1}}\left( x \right)={{\left\| x
\right\|}^{-1}}h_k(x)$. Let $\lambda_{k,m}:=2m+2\left\langle
k\right\rangle+N-2$ and define the Laguerre polynomials as
$$L_l^\mu(t):=\sum_{j=0}^l\frac{{(-1)}^j\Gamma(\mu+l+1)}{(l-j)!\Gamma(\mu+j+1)}\frac{t^j}{j!},\;\mathrm{Re}\mu>-1.$$
In \cite{BSK2} the authors constructed an orthonormal basis
$\left\{\left.\mathrm\Phi_{l,m,j}\right|
l\in\mathbb{N},\;m\in\mathbb{N},\;j=1,2,\cdots,d(m)\right\}$ of
${{L}^{2}}( {{\mathbb{R}}^{N}},$ ${\vartheta_{k,1}}\left( x
\right)dx )$, where
$${\mathrm\Phi}_{l,m,j}\left(x\right):=\left(\frac{2^{\lambda_{k,m}+1}\Gamma(l+1)}{\Gamma(\lambda_{k,m}+l+1)}\right)^{1/2}Y_j^m\left(x\right)L_l^{\lambda_{k,m}}\left(2\left\|x\right\|\right)\exp\left(-\left\|x\right\|\right).$$
They are eigenfunctions for the $(k,1)$-generalized harmonic
oscillator $-{{\Delta
}_{k,1}}=-\left\|x\right\|^{}\triangle_k+\left\|x\right\|$, i.e.,
\begin{equation}\label{Delta}-{{\Delta }_{k,1}}\Phi _{l,m,j}\left( x \right)=\left( 2l+{{\lambda }_{k,m}}+1 \right)\Phi _{l,m,j}\left( x \right).\end{equation}

The $(k,1)$-generalized Laguerre holomorphic semigroup
$e^{z\triangle_{k,1}}\;(\Re z\geq0)$ on ${{L}^{2}}\left(
{{\mathbb{R}}^{N}},{\vartheta_{k,1}}\left( x \right)dx \right)$
has its spectral decomposition (see \cite[(4.3)]{Te})

\begin{align}\label{expand}e^{z\triangle_{k,1}}(f)\left(x\right)=\sum_{l,m,j}e^{-z\left(2l+\lambda_{k,m}+1\right)}{\left\langle f,\Phi_{l,m,j}\right\rangle}_{k,1}\Phi_{l,m,j}\left(x\right),\;f\in {{L}^{2}}\left( {{\mathbb{R}}^{N}},{\vartheta_{k,1}}\left( x \right)dx \right),\end{align}
where ${\left\langle
f,g\right\rangle}_{k,1}=\int_{\mathbb{R}^N}f(x)g(x)\vartheta_{k,1}(x)dx$.
It is a Hilbert--Schmidt operator for $\Re z>0$ and a unitary
operator on $\Re z=0$ (see \cite[Theorem 3.39]{BSK2}). By Schwartz
kernel theorem, the operator $e^{z\triangle_{k,1}}\;(\Re z\geq0)$
has the following integral representation  (see
\cite[(4.56)]{BSK2})
\begin{align}\label{integral}e^{z\triangle_{k,1}}\left(f\right)\left(x\right)=c_{k,1}\int_{\mathbb{R}^N}f\left(y\right)\Lambda_{k,1}\left(x,y;z\right)\vartheta_{k,1}\left(y\right)dy,\end{align}
where
$c_{k,1}=\left(\int_{\mathbb{R}^N}\exp\left(-\left\|x\right\|\right)\vartheta_{k,1}\left(x\right)dx\right)^{-1}$
and
\begin{align}\label{LambdaV}\Lambda_{k,1}\left(r\omega,s\eta;z\right):=\left({\widetilde V}_kh_{k,1}\left(r,s;z;\cdot\right)\right)(\omega,\eta)\end{align}
for $x=r\omega$, $y=s\eta$, $r,s>0$ and $\omega,\eta\in \mathbb
S^{N-1}$. Here ${\widetilde V}_k$ is defined by $\left({\widetilde
V}_kh\right)\left(x,y\right):=\left(V_kh_y\right)(x),$ where
$h_y(\cdot):=h\left(\left\langle\cdot,y\right\rangle\right)$ for a
continuous function $h(t)$ of one variable. And
$h_{k,1}\left(r,s;z;w\right)$ has its closed formula
\begin{align}\label{eqn:4.20}
h_{k,1}(r,s;z;w)=\frac{\exp(-(r+s)\coth(z))}{\sinh(z)^{2\langle
k\rangle+N-1}}\displaystyle
   \Gamma\Bigl(\langle k\rangle+\frac{N-1}{2}\Bigr)
   \widetilde{I}_{\langle k\rangle+\frac{N-3}{2}}
   \Bigl( \frac{\sqrt{2}(rs)^{\frac{1}{2}}}{\sinh z}
          (1+w)^{\frac{1}{2}} \Bigr),
\end{align}
where $\widetilde{I}_{v}$ is the normalized $I$-Bessel function
and has the following integral formula (see, e.g., \cite[6.15
(2)]{Wa})
$${\widetilde I}_v(w)=\frac1{\sqrt{\mathrm\pi}\Gamma\left(\nu+\frac12\right)}\int_{-1}^1e^{wu}\left(1-u^2\right)^{v-{\textstyle\frac12}}du,\;\nu>-1/2,\;w\in\mathbb{C}.$$
The integral on the right hand side of \eqref{integral} converges
absolutely for all $f\in {{L}^{2}}\left(
{{\mathbb{R}}^{N}},{\vartheta_{k,1}}\left( x \right)dx \right)$ if
$\Re z>0$ and for all $f\in {({L}^{1}\cap{L}^{2})}\left(
{{\mathbb{R}}^{N}},{\vartheta_{k,1}}\left( x \right)dx \right)$ if
$\Re z=0$ (see \cite[Corollary 4.28]{BSK2}). From \eqref{LambdaV}
and \eqref{eqn:4.20} we get an expression of
$\Lambda_{k,1}\left(x,y;z\right)$ (a slight modification of
Proposition 5.10 in \cite{BSK2})
\begin{align}\label{eqn:4.21}
\Lambda_{k,1}\left(x,y;z\right)
=&\frac{\exp(-(\left\|x\right\|+\left\|y\right\|)\coth(z))}{\sinh(z)^{2\langle
k\rangle+N-1}}\displaystyle
   \Gamma\Bigl(\langle k\rangle+\frac{N-1}{2}\Bigr)\notag\\&
   \times V_k\left(\widetilde{I}_{\langle k\rangle+\frac{N-3}{2}}
   \Bigl( \frac{1}{\sinh z}
          \sqrt{2(\left\|x\right\|\left\|y\right\|+\left\langle x,\cdot\right\rangle)} \Bigr)\right)\left(y\right).
\end{align}

Let
$$B_{k,1}\left(x,y\right):=e^{i\frac\pi2(2\left\langle k\right\rangle+N-1)}\Lambda_{k,1}\left(x,y;i\frac{\mathrm\pi}2\right).$$
Then the $(k,1)$-generalized Fourier transform on ${{L}^{2}}\left(
{{\mathbb{R}}^{N}},{\vartheta_{k,1}}\left( x \right)dx \right)$
can be expressed as
$$F_{k,1}f\left(\xi\right)=c_{k,1}\int_{\mathbb{R}^N}f\left(y\right)B_{k,1}\left(\xi,y\right)\vartheta_{k,1}\left(y\right)dy,\;\;\xi\in\mathbb{R}^N$$
because $F_{k,1}:=e^{i\frac{\pi}{2}(2\left\langle
k\right\rangle+N-1)}{\mathcal I}_{k,1}\left(\frac{\pi i}2\right)$.
It has the property (see \cite[Theorem 5.3]{BSK2})
$$F_{k,1}^{-1}\left(f\right)=F_{k,1}\left(f\right).$$
The $(k,1)$-generalized translation $\tau_y$ is defined on
${{L}^{2}}\left( {{\mathbb{R}}^{N}},{\vartheta_{k,1}}\left( x
\right)dx \right)$ by (see \cite{BD})
$$F_{k,1}\left(\tau_yf\right)\left(\xi\right):=B_{k,1}\left(y,\xi\right)F_{k,1}\left(f\right)\left(\xi\right),\;\;\xi\in\mathbb{R}^N.$$
It is analogous to the translation operator
$\tau_yf\left(x\right)=f\left(x-y\right)$ in classical Fourier
analysis. The above definition makes sense as $F_{k,1}$ is an
isometry from ${{L}^{2}}\left(
{{\mathbb{R}}^{N}},{\vartheta_{k,1}}\left( x \right)dx \right)$
onto itself. Assume $\langle k\rangle+\frac{N-2}{2}>0$. Then
$\left|B_{k,1}\left(x,y\right)\right|\leq1$ (see
\cite[Theorem 5.11]{BSK2}) and so $\tau_y$ can also be defined as
$$\tau_yf\left(x\right)=c_{k,1}\int_{\mathbb{R}^N}B_{k,1}\left(x,\xi\right)B_{k,1}\left(y,\xi\right)F_{k,1}\left(f\right)\left(\xi\right)\vartheta_{k,1}\left(\xi\right)d\xi$$
for $f\in \mathcal L_k^1\left(\mathbb{R}^N\right)$, where
$\mathcal L_k^1\left(\mathbb{R}^N\right):=\left\{f\in
L^1\left(\mathbb{R}^N,\vartheta_{k,1}\left(x\right)dx\right):\;F_{k,1}\left(f\right)\in
L^1\left(\mathbb{R}^N,\vartheta_{k,1}\left(x\right)dx\right)\right\}.$
This formula holds true on Schwartz space $\mathcal S\left(\mathbb{R}^N\right)$ since $\mathcal S\left(\mathbb{R}^N\right)$ is a subspace of $\mathcal L_k^1\left(\mathbb{R}^N\right)$. The operator $\tau_y$ satisfies the following properties:\\
(1). For every $x,y\in \mathbb R^N$,
\begin{align}\label{prop1}\tau_yf\left(x\right)=\tau_xf\left(y\right),\;f\in\mathcal
S\left(\mathbb{R}^N\right).\end{align} (2). For every $y\in
\mathbb R^N$,
\begin{align}\label{prop2}\int_{\mathbb{R}^N}\tau_yf\left(x\right)g\left(x\right)\vartheta_{k,1}\left(x\right)dx=\int_{\mathbb{R}^N}f\left(x\right)\tau_yg\left(x\right)\vartheta_{k,1}\left(x\right)dx,\;\;f,g\in\mathcal S\left(\mathbb{R}^N\right).\end{align}
Here the property (1) corresponds to
$\tau_yf\left(x\right)=\tau_{-x}f\left(-y\right)$ and (2)
corresponds to the skew-symmetry in classical Fourier analysis and
Dunkl analysis.

For any radial function $f\in\mathcal S\left(\mathbb{R}^N\right)$,
i.e.,  $f(x)=f_0\left(\left\|x\right\|\right)$, $\langle
k\rangle+\frac{N-2}{2}>0$, $\tau_y$ can be expressed as follows
(see \cite{BD})
\begin{align}\label{radial}\tau_yf(x)=&\frac{\Gamma\left(\frac{N-1}2+\left\langle k\right\rangle\right)}{\sqrt\pi\Gamma\left(\frac{N-2}2+\left\langle k\right\rangle\right)}\times\notag\\&V_k\left(\int_{-1}^1f_0\left(\left\|x\right\|+\left\|y\right\|-\sqrt{2\left(\left\|x\right\|\left\|y\right\|+\left\langle\cdot,y\right\rangle\right)}u\right)\left(1-u^2\right)^{\frac N2+\left\langle k\right\rangle-2}du\right)\left(x\right).\end{align}
And so $\tau_y$ is positive on radial functions and can be
extended as a bounded operator to the space of all radial
functions on ${{L}^{p}}\left(
{{\mathbb{R}}^{N}},{\vartheta_{k,1}}\left( x \right)dx
\right),\;1\leq p\leq2$. Further, if $f$ is a nonnegative radial
function on ${{L}^{1}}\left(
{{\mathbb{R}}^{N}},{\vartheta_{k,1}}\left( x \right)dx \right)$,
then
\begin{align}\label{inttauye}\int_{\mathbb{R}^N}\tau_yf\left(x\right)\vartheta_{k,1}\left(x\right)dx=\int_{\mathbb{R}^N}f\left(x\right)\vartheta_{k,1}\left(x\right)dx.\end{align}
The authors in \cite{BD} also gave a special case of the formula
for radial functions
\begin{align}\label{tauye}
\tau_y\left(e^{-\lambda\left\|\cdot\right\|}\right)\left(x\right)
=\displaystyle
   \Gamma\Bigl(\langle k\rangle+\frac{N-1}{2}\Bigr)e^{-\lambda(\left\|x\right\|+\left\|y\right\|)}
   V_k\left(\widetilde{I}_{\langle k\rangle+\frac{N-3}{2}}
   \Bigl( \lambda
          \sqrt{2(\left\|x\right\|\left\|y\right\|+\left\langle x,\cdot\right\rangle)} \Bigr)\right)\left(y\right).
\end{align}

\section{H\"ormander type condition}
Let $(X,d)$ be a metric space. Denote $B(x,r)$ to be the ball
$B\left(x,r\right):=\left\{y\in X:d\left(x,y\right)\leq r\right\}$
for $x\in X$. If there exists a doubling measure $m$, i.e., there
exists a measure $m$ such that for some absolute constant $C$,
\begin{align}\label{doubling}m\left(B\left(x,2r\right)\right)\leq Cm\left(B\left(x,r\right)\right),\;\forall x\in\mathbb{R}^N,\;r>0,\end{align}
then $(X,d)$ is a space of homogeneous type. The
Calder\'on--Zygmund theory on a space of homogeneous type
$(X,d,m)$ says that for $f\in L^1(X,m)\cap L^2(X,m)$ and
$\lambda>\frac{{\left\|f\right\|}_1}{m\left(X\right)}$, there
exists the Calder\'on--Zygmund decomposition $f=h+b$ with
$b=\sum_j b_j$ and a sequence of balls $(B(y_j,r_j))_j$ =$(B_j)_j$
such that for some absolute constant $C$,

\begin{itemize}
  \item [(i)]$\left\|h \right\|_\infty\leq C \lambda$;
  \item [(ii)]$supp(b_j)\subset B_j$;
  \item [(iii)]$\displaystyle{\int_{B_j}b_j(x)dm(x)=0}$;
\item [(iv)]$\left\|b_j\right\|_{L^1\left(X,m\right)}\leq
C\;\lambda\, m(B_j)$; \item [(v)]$\displaystyle{\sum_j m(B_j)\leq
C\;\frac{\left\|f\right\|_{L^1\left(X,m\right)}}{\lambda}}$ .
\end{itemize}
From the Calder\'on--Zygmund decomposition one can deduce that for
a bounded operator $S$ on $L^2(X,m)$ associated with kernel
$K(x,y)$, if $K(x,y)$ satisfies a H\"ormander type condition, then
the operator $S$ can be extended to a bounded operator on
$L^p(X,m)$ $(1<p\leq 2)$ and a weakly bounded operator on
$L^1(X,m)$. We refer to \cite[Chapter III]{CG} for this theory.

Now we adapt Calder\'on--Zygmund theory to the $(k,1)$-generalized
setting by constructing the metric space corresponding to this setting first. For $x,y\in \mathbb
R^N$, define a function $d$ from $\mathbb{R}^N\times\mathbb{R}^N$ to $\mathbb R$ as
\begin{align*}d\left(x,y\right):&=\sqrt{\left\|x\right\|+\left\|y\right\|-\sqrt{2\left(\left\|x\right\|\left\|y\right\|+\left\langle x,y\right\rangle\right)}}\\&=\sqrt{\left\|x\right\|+\left\|y\right\|-2\sqrt{\left\|x\right\|\left\|y\right\|}\cos\frac\theta2}\geq\left|\sqrt{\left\|x\right\|}-\sqrt{\left\|y\right\|}\right|,\end{align*}
where $\theta=\arccos\frac{\left\langle
x,y\right\rangle}{\left\|x\right\|\left\|y\right\|}$,
$0\leq\theta\leq\pi$. We need to equip $\mathbb{R}^N$ with this function as the metric in $(k,1)$-generalized analysis in view of the expression \eqref{radial} of $(k,1)$-generalized translation operators.

\begin{prop}\label{d(x,y)}
The function $d\left(x,y\right)$ is a metric.
\end{prop}

\begin{proof}
The symmetry property is obvious. For the positivity property, if
$d\left(x,y\right)=0$, then $\left\|x\right\|=\left\|y\right\|$
and $\sqrt{\left\|x\right\|-\left\|x\right\|\cos\frac\theta2}=0$
leading to $\theta=0$. Hence $x=y$.

Then we turn to prove the triangle inequality. Let
$$\alpha=\arccos\frac{\left\langle x,y\right\rangle}{\left\|x\right\|\left\|y\right\|},\; \beta=\arccos\frac{\left\langle x,z\right\rangle}{\left\|x\right\|\left\|z\right\|},\; \gamma=\arccos\frac{\left\langle z,y\right\rangle}{\left\|z\right\|\left\|y\right\|},\; 0\leq\alpha,\;\beta,\;\gamma\leq\pi.$$ Then we have $\beta+\gamma \geq\alpha$ from the triangle inequality of the spherical distance. Therefore,
$$d\left(x,y\right)\leq\sqrt{\left\|x\right\|+\left\|y\right\|-2\sqrt{\left\|x\right\|\left\|y\right\|}\cos\frac{\beta+\gamma}2}.$$
It suffices to show that
$$\sqrt{\left\|x\right\|+\left\|y\right\|-2\sqrt{\left\|x\right\|\left\|y\right\|}\cos\frac{\beta+\gamma}2}\leq d\left(x,z\right)+d\left(z,y\right).$$
Take the square of the above inequality and eliminate some items.
It suffices to show the following inequality,
\begin{align*}\left\|x\right\|\left\|y\right\|\sin^2\frac{\beta+\gamma}2&+\left\|x\right\|\left\|z\right\|\sin^2\frac\beta2+\left\|z\right\|\left\|y\right\|\sin^2\frac\gamma2+2\left\|z\right\|\sqrt{\left\|x\right\|\left\|y\right\|}\cos\frac\beta2\cos\frac\gamma2\\&+2\left\|x\right\|\sqrt{\left\|z\right\|\left\|y\right\|}\cos\frac\beta2\cos\frac{\beta+\gamma}2+2\left\|y\right\|\sqrt{\left\|z\right\|\left\|x\right\|}\cos\frac\gamma2\cos\frac{\beta+\gamma}2\\\geq&2\left\|x\right\|\sqrt{\left\|z\right\|\left\|y\right\|}\cos\frac\gamma2+2\left\|y\right\|\sqrt{\left\|z\right\|\left\|x\right\|}\cos\frac\beta2+2\left\|z\right\|\sqrt{\left\|x\right\|\left\|y\right\|}\cos\frac{\beta+\gamma}2.\end{align*}
And the inequality is equivalent to
$$\left(\sqrt{\left\|x\right\|\left\|y\right\|}\sin\frac{\beta+\gamma}2-\sqrt{\left\|x\right\|\left\|z\right\|}\sin\frac\beta2-\sqrt{\left\|z\right\|\left\|y\right\|}\sin\frac\gamma2\right)^2\geq0.$$
Proposition \ref{d(x,y)} is therefore proved.
\end{proof}
\begin{rem}
i). For the one dimensional case, the metric $d(x,y)$ recedes to $$d(x,y)=\left\{\begin{array}{l}\sqrt{\left|x-y\right|},\;xy\leq0\\\left|\sqrt{\left|x\right|}-\sqrt{\left|y\right|}\right|,\;xy>0\end{array}\right..$$
The ball with respect to this metric in this case was already used in \cite{BD2} to define the generalized Hardy--Littlewood maximal operator.\\
ii). A continuous rectifiable curve between two distinct points does not
necessarily exist with respect to this metric. For example, if
we take $x=-1$ and $y=1$ for the one dimensional case, then distance between $x$ and $y$ with respect to the induced length metric is no less than $\underset
n{\sup}\sum_{i=1}^n\sqrt{\frac2n}=\infty$.
\end{rem}

\begin{prop}
$\left(\mathbb{R}^N,d\right)$ is a complete metric space.
\end{prop}

\begin{proof}
We will show that $d(x,y)$ is equivalent to the Euclidean metric.
If $y_n\rightarrow y$ with respect to the Euclidean metric, then
$d\left(y_n,y\right)\rightarrow0$ obviously. If
$d\left(y_n,y\right)\rightarrow0$, then
$\left\|y_n\right\|\rightarrow\left\|y\right\|$. Denote
$\theta_n=\arccos\frac{\left\langle
y_n,y\right\rangle}{\left\|y_n\right\|\left\|y\right\|}$. Then
$$\sqrt{2\left\|y\right\|-2\left\|y\right\|\lim_{n\rightarrow\infty}\cos\frac{\theta_n}2}=0.$$
So, $\lim_{n\rightarrow\infty}\cos\theta_n=1$ and
$\lim_{n\rightarrow\infty}\left\langle
y_n,y\right\rangle=\left\|y\right\|^2$. Hence
\begin{align*}\lim_{n\rightarrow\infty}\left\|y_n-y\right\|=\lim_{n\rightarrow\infty}\sqrt{\left\|y_n\right\|^2+\left\|y\right\|^2-2\left\langle y_n,y\right\rangle}=0.\tag*{\qedhere}\end{align*}
\end{proof}

The closure of an open ball in a metric space is not necessarily
the closed ball. In \cite {Wo} the authors gave a sufficient but
not necessary condition such that the closure of the open ball is
the closed ball. They showed that if the metric is weakly convex,
i.e., for any two different points $x$ and $y$, there exists
$z\neq x,y$, such that $d(x,y)=d(x,z)+d(z,y)$. The metric $d$ we
are concerned with is not weakly convex obviously but the closure
of the open ball with respect to this metric is still the closed
ball.
\begin{thm}
The closure $\overline{B_0\left(x,r\right)}$ of the open ball
$B_0\left(x,r\right)=\left\{y:d\left(y,x\right)<r\right\},\;r>0$
is the closed ball $B(x,r)$.
\end{thm}

\begin{proof}
Let $y$ be a point $\mathbb R^N$ distinct from $x$ such that
$d(x,y)=r$. We show that for any $\varepsilon>0$, there exists
$z\in B(y,\varepsilon)$, such that $d(x,z)<r=d(x,y)$. Let
$M_{x},\,x\neq0$ be the mapping $M_{x}:\;\mathbb R^N
\rightarrow\lbrack0,+\infty),\;\;y\mapsto d(x,y)$ and
$L_{x}(y):=M_{x}{(y)}^2$. It suffices to show that the function
$L_{x}$ takes no minimum point on $\mathbb R^N$ except at $y=x$.
Notice that $L_{x}$ is differentiable on
$\mathbb{R}^N\backslash\left\{0\right\}$. We calculate the points
such that
$$0=\frac{\partial L_{x}}{\partial y_i}=\frac{y_i}{\left\|y\right\|}-\frac{\left\|x\right\|\frac{y_i}{\left\|y\right\|}+x_i}{\sqrt{2\left(\left\|x\right\|\left\|y\right\|+\left\langle x,y\right\rangle\right)}},\;\;i=1,2,\cdots ,N.$$
By summing up the square, we get
$$\sqrt{2\left(\left\|x\right\|\left\|y\right\|+\left\langle x,y\right\rangle\right)}=2\left\|y\right\|\;\mathrm{and}\;x_i=ty_i,\;\mathrm{where}\;t=2-\frac{\left\|x\right\|}{\left\|y\right\|}.$$
Thus $2y_i=\left|t\right|y_i+ty_i$ and $y=x$. For the point $y=0$,
consider the function
\begin{align*}L\left(y_1\right):=L_{x}\left(y_1,0,\dots,0\right)&=\left\|x\right\|+\left|y_1\right|-\sqrt{2\left(\left\|x\right\|\left|y_1\right|+x_1y_1\right)}\\&=\left\{\begin{array}{l}\left\|x\right\|+y_1-\sqrt{2\left(\left\|x\right\|+x_1\right)y_1},\;\;y_1\geq0\\\left\|x\right\|-y_1-\sqrt{-2\left(\left\|x\right\|-x_1\right)y_1},\;\;y_1<0.\end{array}\right.\end{align*}
It does not take minimum at $y_1=0$ obviously. Therefore, $L_{x}$
takes no minimum point on $\mathbb R^N$ except at $y=x$.
\end{proof}

The metric space $\left(\mathbb{R}^N,d\right)$, rather than the
standard Euclidean metric space, is the natural metric space
corresponding to the $(k,1)$-generalized setting when metric is
involved due to the expression of the $(k,1)$-generalized
translation operators. In the following theorem we give a
characterization of support of the $(k,1)$-generalized translation
of a function supported in
$B(0,r)=\left\{y\in\mathbb{R}^N:\sqrt{\left\|y\right\|}\leq
r\right\}$.
\begin{thm}
Let $f=f_0\left(\left\|\cdot\right\|\right)$ be a nonnegative
radial function on ${{L}^{2}}\left(
{{\mathbb{R}}^{N}},{\vartheta_{k,1}}\left( x \right)dx \right)$,
$\supp f=B(0,r)$, then $$\mathrm{supp}\tau_xf=\bigcup_{g\in
G}B(gx,r).$$
\end{thm}

\begin{proof}
We extend the formula of $(k,1)$-generalized translations on
radial Schwartz functions \eqref{radial} to all continuous radial
functions on ${{L}^{2}}\left(
{{\mathbb{R}}^{N}},{\vartheta_{k,1}}\left( x \right)dx \right)$
first. The proof goes similar as the Lemma 3.4 in \cite{}. The
only difference is to take the set $A_n$ in the proof as
$$A_n\equiv A_n(y):=\left\{x\in\mathbb{R}^N:2^{-n}\leq\left|\sqrt{\left\|x\right\|}-\sqrt{\left\|y\right\|}\right|\leq\sqrt{\left\|x\right\|}+\sqrt{\left\|y\right\|}\leq2^n\right\}$$
for $n\in\mathbb{N}$ and
$n\geq{\textstyle\frac12}\left[\log\;\left\|y\right\|/\log\;2\right]+1$,
since
$$\left|\sqrt{\left\|x\right\|}-\sqrt{\left\|y\right\|}\right|\leq\sqrt{\left\|x\right\|+\left\|y\right\|-\sqrt{2\left(\left\|x\right\|\left\|y\right\|+\left\langle\eta,y\right\rangle\right)}u}\leq\sqrt{\left\|x\right\|}+\sqrt{\left\|y\right\|}$$
for $\eta\in co(G.x)$ and $u\in\left[-1,1\right]$.

Then we prove the theorem for continuous nonnegative radial
functions. For the proof of $supp\tau_xf\subseteq\bigcup_{g\in
G}B(gx,r)$, from the radial formula \eqref{radial} of
$(k,1)$-generalized translations and notice that for any $\eta\in
co(G.x)$ and $u\in[-1,1]$,
\begin{align}\label{min}\sqrt{\left\|x\right\|+\left\|y\right\|-\sqrt{2\left(\left\|x\right\|\left\|y\right\|+\left\langle\eta,y\right\rangle\right)}u}\geq \underset{g\in G}{\min}\;d\left(gx,y\right),\end{align}
we have $\tau_xf(y)=0$ for $y\in\left(\bigcup_{g\in
G}B(gx,r)\right)^c$ if $\mathrm{supp}f\subseteq B(0,r)$. For the
converse part $\bigcup_{g\in G}B(gx,r)\subseteq\supp\tau_xf$, we
will show that $\bigcup_{g\in G}B_0(gx,r)\subseteq\supp\tau_xf$
first. Suppose there exists a $y\in\bigcup_{g\in G}B_0(gx,r)$ for
which $y\not\in\supp\tau_xf$. Then there exists $\varepsilon>0$,
such that for any $z\in B(y,\varepsilon)$, we have
$z\in\bigcup_{g\in G}B_0(gx,r)$ (that is, there also exists a
$g\in G$ such that $d\left(z,gx\right)< r$) and
\begin{align*}0=\tau_xf(z)=&\frac{\Gamma\left(\frac{N-1}2+\left\langle k\right\rangle\right)}{\sqrt\pi\Gamma\left(\frac{N-2}2+\left\langle k\right\rangle\right)}\times\\&\int_{\mathbb{R}^N}\int_{-1}^1f_0\left(\left\|x\right\|+\left\|z\right\|-\sqrt{2\left(\left\|x\right\|\left\|z\right\|+\left\langle\eta,z\right\rangle\right)}u\right)\left(1-u^2\right)^{\frac N2+\left\langle k\right\rangle-2}dud\mu_x\left(\eta\right).\end{align*}
Thus
$$f_0\left(\left\|x\right\|+\left\|z\right\|-\sqrt{2\left(\left\|x\right\|\left\|z\right\|+\left\langle\eta,z\right\rangle\right)}u\right)=0$$
for any $\eta\in \supp\mu_x$ and $u\in[-1,1]$. Then from a result
of Gallardo and Rejeb (see \cite{Ga}), that the orbit of $x$,
$G.x$, is contained in $\supp\mu_x$, we can select   $u=1$ and
$\eta=gx$ for the above $g$. Then we get
$f_0\left(d\left(gx,z\right)^2\right)=0$ for all $z\in
B(y,\varepsilon)$. But $d\left(z,gx\right)< r$, which contradicts
to that $\supp f_0=[0,r^2]$. Then from Theorem 3.4, we get
$\bigcup_{g\in G}B(gx,r)\subseteq\supp\tau_xf$.

The conclusion for all nonnegative radial function on
${{L}^{2}}\left( {{\mathbb{R}}^{N}},{\vartheta_{k,1}}\left( x
\right)dx \right)$ can then be derived from the density of
continuous functions with compact support $B(0,r)$ in
${{L}^{2}}\left( B(0,r),{\vartheta_{k,1}}\left( x \right)dx
\right)$ and the positivity of the $(k,1)$-generalized
translations on radial functions as in \cite[Theorem 1.2]{Te1}.
\end{proof}

Denote by $dm_{k,1}(x)=\vartheta_{k,1}\left(x\right)dx$. The
measure $m_{k,1}$ satisfies the scaling property
\begin{align}\label{scaling}m_{k,1}\left(B\left(tx,\sqrt tr\right)\right)=t^{2\left\langle k\right\rangle+N-1}m_{k,1}\left(B\left(x,r\right)\right),\;t>0.\end{align}
From polar coordinate transformation we have
\begin{align*}m_{k,1}\left(B(x,r)\right)&=\underset{\rho+\left\|x\right\|-2\sqrt{\rho\left\|x\right\|}\cos\frac\theta2\leq r^2}{\int_{\mathbb S^{N-1}}\int_{\left(0,+\infty\right)}^{}}\;\rho^{2\left\langle k\right\rangle+N-2}d\rho h_k\left(\omega\right)d\omega\\&\overset{u=\sqrt\rho}=\underset{u^2+\left\|x\right\|-2u\sqrt{\left\|x\right\|}\cos\frac\theta2\leq r^2}{\int_{\mathbb S^{N-1}}\int_{\left(0,+\infty\right)}^{}}\;u^{2(2\left\langle k\right\rangle+N)-3}duh_k\left(\omega\right)d\omega\\&\overset{z=u\omega}=\int_{E\left(x_\omega,r\right)}\left\|z\right\|^{2\left\langle k\right\rangle+N-2}h_k\left(z\right)dz,\end{align*}
where $$\theta=\arccos\frac{\left\langle
x,\omega\right\rangle}{\left\|x\right\|},\;
x_\omega=\sqrt{\left\|x\right\|}\frac{x+\left\|x\right\|\omega}{\left\|x+\left\|x\right\|\omega\right\|},$$
and $E\left(x_\omega,r\right)$ denotes the Euclidean ball centered
at $x_\omega$ with radius $r$. For the one dimensional case, this expression coincides that of the measure of the ball in the proof of Lemma 2.2 in \cite{BD2}.
So if $2\left\langle
k\right\rangle+N-2> 0$, then for any $x\in\mathbb R^N$ and $r>0$,
$m_{k,1}\left(B\left(x,r\right)\right)$ is finite and
$m_{k,1}\left(B(tx,r)\right)$ is nondecreasing as $t$ grows.
  It is then easy to check that $m_{k,1}$ is a doubling measure when $2\left\langle k\right\rangle+N-2> 0$ combining \eqref{scaling}.
Therefore, $(\mathbb R^N,d,m_{k,1})$ is a space of homogeneous
type and for all $f\in
L^1\left(\mathbb{R}^N,\vartheta_{k,1}\left(x\right)dx\right)\cap
L^2\left(\mathbb{R}^N,\vartheta_{k,1}\left(x\right)dx\right)$ and
$\lambda>0$, there exists the corresponding Calder\'on--Zygmund
decomposition of $f$ satisfying $(i)$--$(v)$.

Define the distance between the two orbits $G.x$ and $G.y$ as
$d_G\left(x,y\right)=\underset{g\in G}{\min}\;d\left(gx,y\right).$
Now we are ready to give the H\"ormander type condition adapted to
$(k,1)$-generalized setting. It is a modification of the H\"ormander type condition
on a
homogeneous space in \cite[Chapter III, Theorem 2.4]{CG} because the Calder\'on--Zygmund theory on homogeneous spaces cannot be applied to this setting. We omit the proof because it is similar to that of Theorem 3.1 in \cite{AS}.

\begin{thm}\label{Hor}
For $2\left\langle k\right\rangle+N-2> 0$, let $K$ be a measurable
function on
$\mathbb{R}^N\times\mathbb{R}^N\backslash\left\{\left(x,g.x\right);x\in\mathbb{R}^N,\;g\in
G\right\}$ and $S$ be a bounded operator on ${{L}^{2}}\left(
{{\mathbb{R}}^{N}},{\vartheta_{k,1}}\left( x \right)dx \right)$
associated with the kernel $K$ such that for any compactly
supported function $f\in {{L}^{2}}\left(
{{\mathbb{R}}^{N}},{\vartheta_{k,1}}\left( x \right)dx \right)$,
$$S\left(f\right)\left(x\right)=\int_{\mathbb{R}^N}K\left(x,y\right)f\left(y\right)\vartheta_{k,1}\left(y\right)dy,\;G.x\cap \supp f=\varnothing.$$
If $K$ satisfies
$$\int_{d_G(x,y)>2d(y,y_0)}\left|K(x,\;y)-K(x,\;y_0)\right|\vartheta_{k,1}\left(x\right)dx\leq C,\;\;y,y_0\in\mathbb{R}^N,$$
then $S$ extends to a bounded operator on ${{L}^{p}}\left(
{{\mathbb{R}}^{N}},{\vartheta_{k,1}}\left( x \right)dx \right)$
for $1<p\leq 2$ and a weakly bounded operator on ${{L}^{1}}\left(
{{\mathbb{R}}^{N}},{\vartheta_{k,1}}\left( x \right)dx \right)$.
\end{thm}

\section{Main Results}
We reformulate the reproducing kernel $\Lambda_{k,1}\left(x,y;z\right)$ of
$e^{z\triangle_{k,1}}$ from \eqref{eqn:4.21}, \eqref{tauye} and
\eqref{radial} as
\begin{align}\label{Lambda0}\Lambda_{k,1}\left(x,y;z\right)&=\frac1{(\sinh\;z)^{2\langle k\rangle+N-1}}e^{-\tanh\frac z2\left(\left\|x\right\|+\left\|y\right\|\right)}\tau_y\left(e^{-\frac1{\sinh z}\left\|\cdot\right\|}\right)(x)\\&\notag=\frac{\Gamma\left(\frac{N-1}2+\left\langle k\right\rangle\right)}{\sqrt\pi\Gamma\left(\frac{N-2}2+\left\langle k\right\rangle\right)}\frac1{(\sinh\;z)^{2\langle k\rangle+N-1}}e^{-\tanh\frac z2\left(\left\|x\right\|+\left\|y\right\|\right)}\\&\notag\;\;\;\;\cdot V_k\bigg(\int_{-1}^1e^{-\frac1{\sinh z}\left(\left\|x\right\|+\left\|y\right\|-\sqrt{2\left(\left\|x\right\|\left\|y\right\|+\left\langle\cdot,y\right\rangle\right)}u\right)}\left(1-u^2\right)^{\frac N2+\left\langle k\right\rangle-2}du\bigg)\left(x\right).\end{align}
Let $z=t$, $t>0$. For $0<t\leq 1$, $\sinh\;t$ behaves like $t$. So
\begin{align}\label{Lambda1}\left|\Lambda_{k,1}\left(x,y;t\right)\right|\leq C\frac1{t{}^{2\langle k\rangle+N-1}}\tau_y\left(e^{-\frac b{t}\left\|\cdot\right\|}\right)(x).\end{align}
For $t>1$, $\sinh\;t$ behaves like $e^t$. So
\begin{align}\label{Lambda2}\left|\Lambda_{k,1}\left(x,y;t\right)\right|\leq Ce^{-\left({2\langle k\rangle+N-1}\right)t}\tau_y\left(e^{-b\left\|\cdot\right\|}\right)(x).\end{align}

From \eqref{Delta} we can define the imaginary powers
$\left(-\triangle_{k,1}\right)^{-i\sigma},\sigma\in\mathbb{R}$ for
$f\in
L^2\left(\mathbb{R}^N,\vartheta_{k,1}\left(x\right)dx\right)$ of
the $(k,1)$-generalized harmonic oscillator $-\triangle_{k,1}$
naturally as
\begin{align}\label{imaginary}\left(-\triangle_{k,1}\right)^{-i\sigma}(f)\left(x\right)=\sum_{l,m,j}\left(2l+\lambda_{k,m}+1\right)^{-i\sigma}{\left\langle f,\Phi_{l,m,j}\right\rangle}_{k,1}\Phi_{l,m,j}\left(x\right).\end{align}
It is obviously a bounded operator on
$L^2\left(\mathbb{R}^N,\vartheta_{k,1}\left(x\right)dx\right)$
from its spectrum.

In what follow we put
\begin{align}\label{Kj}
K\left(x,y\right)=\int_0^\infty\Lambda_{k,1}\left(x,y;t\right)t^{i\sigma-1}dt.
\end{align}
It is then easy to verify the integral \eqref{Kj} converges
absolutely and that
$\left|K\left(x,y\right)\right|\leq C\frac1{d_G\left(x,y\right)^{2\left(2\langle k\rangle+N-1\right)}}$ for all $x,y\in \mathbb R^N$, $y\not\in G.x$, $2\left\langle
k\right\rangle+N-2> 0$.

Based on the formula
$$\lambda^{-i\sigma}=\frac1{\Gamma\left(i\sigma\right)}\int_0^\infty e^{-t\lambda}t^{i\sigma-1}dt,\;\;\;\lambda>0$$
and \eqref{imaginary}, \eqref{expand}, \eqref{integral}, we can
write $\left(-\triangle_{k,1}\right)^{-i\sigma}$ in the following
way (such definition goes back to \cite{NS2} and \cite{ST})
\begin{align*}\left(-\triangle_{k,1}\right)^{-i\sigma}\left(f\right)\left(x\right)&=\frac1{\Gamma\left(i\sigma\right)}\int_0^\infty e^{t\triangle_{k,1}}\left(f\right)\left(x\right)t^{i\sigma-1}dt\\&=\frac{c_{k,1}}{\Gamma\left(i\sigma\right)}\int_0^\infty t^{i\sigma-1}dt\int_{\mathbb{R}^N}f\left(y\right)\Lambda_{k,1}\left(x,y;t\right)\vartheta_{k,1}\left(y\right)dy.\end{align*}
We can observe that this integral converges absolutely for all compactly supported
functions $f\in {{L}^{2}}(
{{\mathbb{R}}^{N}},{\vartheta_{k,1}}\left( x \right)$ $dx )$
with $\supp f\cap G.x =\varnothing$.
And for compactly supported functions $f\in {{L}^{2}}\left(
{{\mathbb{R}}^{N}},{\vartheta_{k,1}}\left( x \right)dx \right)$,
$G.x\cap \supp f=\varnothing$,
$\left(-\triangle_{k,1}\right)^{-i\sigma}$ satisfies
$$\left(-\triangle_{k,1}\right)^{-i\sigma}\left(f\right)\left(x\right)=\frac{c_{k,1}}{\Gamma\left(i\sigma\right)}\int_{\mathbb{R}^N}K\left(x,y\right)f\left(y\right)\vartheta_{k,1}\left(y\right)dy$$
by changing the order of integration. We will show that the kernel
$K\left(x,y\right)$ of $\left(-\triangle_{k,1}\right)^{-i\sigma}$
satisfies the condition in Theorem \ref{Hor} to prove the
following main theorem.
\begin{thm}\label{main}
For $2\left\langle k\right\rangle+N-2> 0$, the imaginary powers
$\left(-\triangle_{k,1}\right)^{-i\sigma},\sigma\in\mathbb{R}$ of
the $(k,1)$-generalized harmonic oscillator $-\triangle_{k,1}$ are
bounded operators on ${{L}^{p}}\left(
{{\mathbb{R}}^{N}},{\vartheta_{k,1}}\left( x \right)dx \right)$,
$1<p<\infty$ and weakly bounded on ${{L}^{1}}\left(
{{\mathbb{R}}^{N}},{\vartheta_{k,1}}\left( x \right)dx \right)$.
\end{thm}

\section{Proof of Theorem \ref{main} }
We start the proof with two lemmas. The first one is an
enhancement of the triangle inequality of the metric $d(x,y)$.
\begin{lem}\label{qusitrian}
For $u\in [-1,1]$, $\eta\in co(G.x)$, and $x,y\in\mathbb R^N$,
$$\left|\sqrt{\left\|x\right\|+\left\|y\right\|-\sqrt{2\left(\left\|x\right\|\left\|y\right\|+\left\langle \eta,y\right\rangle\right)}u}-\sqrt{\left\|x\right\|+\left\|z\right\|-\sqrt{2\left(\left\|x\right\|\left\|z\right\|+\left\langle \eta,z\right\rangle\right)}u}\right|\leq d\left(y,z\right).$$
\end{lem}

\begin{proof}
If $\eta\in co(G.x)$, then there exists a rotation transformation
$T$ such that $\eta=kT(x),\;0\leq k\leq 1$. So we assume $\eta=kx$
in the proof since
$\left\|T\left(x\right)\right\|=\left\|x\right\|$. Then
$$\sqrt{2\left(\left\|x\right\|\left\|y\right\|+\left\langle\eta,y\right\rangle\right)}u=2\sqrt{\left\|x\right\|\left\|y\right\|}\sqrt{\frac12\left(1+k\cos\alpha\right)}u;$$

$$\sqrt{2\left(\left\|x\right\|\left\|z\right\|+\left\langle\eta,z\right\rangle\right)}u=2\sqrt{\left\|x\right\|\left\|z\right\|}\sqrt{\frac12\left(1+k\cos\beta\right)}u.$$
Denote by
$$\alpha_k=2\arccos\sqrt{\frac12\left(1+k\cos\alpha\right)},\;\;\beta_k=2\arccos\sqrt{\frac12\left(1+k\cos\beta\right)}.$$
We assert that
\begin{align}\label{alphak}\left|\alpha_k-\beta_k\right|\leq\gamma.\end{align}
Here $\alpha,\;\beta$ and $\gamma$ are given as in the proof of
Proposition \ref{d(x,y)}.

Assume $\left\|x\right\|=\left\|y\right\|=\left\|z\right\|=1$.
Then \eqref{alphak} is equivalent to
$$1-\left\langle y,z\right\rangle^2-k^2\left\langle x,y\right\rangle^2-k^2\left\langle x,z\right\rangle^2+2k^2\left\langle x,y\right\rangle\left\langle y,z\right\rangle\left\langle x,z\right\rangle\geq0.$$
It suffices to show that
$$k^2\left(1-\left\langle y,z\right\rangle^2\right)-k^2\left\langle x,y\right\rangle^2-k^2\left\langle x,z\right\rangle^2+2k^2\left\langle x,y\right\rangle\left\langle y,z\right\rangle\left\langle x,z\right\rangle\geq0.$$
And it is equivalent to
$$\det\begin{bmatrix}1&\left\langle x,y\right\rangle&\left\langle x,z\right\rangle\\\left\langle y,x\right\rangle&1&\left\langle y,z\right\rangle\\\left\langle z,x\right\rangle&\left\langle z,y\right\rangle&1\end{bmatrix}\geq0.$$
It is the determinant of the Gram matrix of the three vectors $x$,
$y$, and $z$. Thus \eqref{alphak} is proved.

From assertion \eqref{alphak}, similar to the proof in Proposition
\ref{d(x,y)}, it suffices to show that
$$\left|\sqrt{\left\|x\right\|+\left\|y\right\|-2\sqrt{\left\|x\right\|\left\|y\right\|}u\cos\frac{\alpha_k}{2}}-\sqrt{\left\|x\right\|+\left\|z\right\|-2\sqrt{\left\|x\right\|\left\|z\right\|}u\cos\frac{\beta_k}2}\right|\leq\sqrt{\left\|y\right\|+\left\|z\right\|-2\sqrt{\left\|y\right\|\left\|z\right\|}\cos\frac{\alpha_k-\beta_k}2}.$$
And it suffices to show
\begin{align*}\left\|x\right\|\left\|y\right\|\left(1-u^2\cos^2\frac{\alpha_k}2\right)&+\left\|x\right\|\left\|z\right\|\left(1-u^2\cos^2\frac{\beta_k}2\right)+\left\|z\right\|\left\|y\right\|\sin^2\frac{\alpha_k-\beta_k}2\\&+2\left\|z\right\|\sqrt{\left\|x\right\|\left\|y\right\|}u\cos\frac{\beta_k}2\cos\frac{\alpha_k-\beta_k}2\\&+2\left\|x\right\|\sqrt{\left\|z\right\|\left\|y\right\|}u^2\cos\frac{\beta_k}2\cos\frac{\alpha_k}2+2\left\|y\right\|\sqrt{\left\|z\right\|\left\|x\right\|}u\cos\frac{\alpha_k-\beta_k}2\cos\frac{\alpha_k}2\\\geq2&\left\|x\right\|\sqrt{\left\|z\right\|\left\|y\right\|}\cos\frac{\alpha_k-\beta_k}2+2\left\|y\right\|\sqrt{\left\|z\right\|\left\|x\right\|}u\cos\frac{\beta_k}2+2\left\|z\right\|\sqrt{\left\|x\right\|\left\|y\right\|}u\cos\frac{\alpha_k}2.\end{align*}
The above is equivalent to
\begin{align*}\bigg(\sqrt{\left\|x\right\|\left\|y\right\|}u\sin\frac{\alpha_k}2-\sqrt{\left\|x\right\|\left\|z\right\|}u\sin\frac{\beta_k}2&-\sqrt{\left\|z\right\|\left\|y\right\|}\sin\frac{\alpha_k-\beta_k}2\bigg)^2+\left\|x\right\|\left(\left\|y\right\|+\left\|z\right\|\right)\left(1-u^2\right)\\&\geq2\left\|x\right\|\sqrt{\left\|z\right\|\left\|y\right\|}\left(1-u^2\right)\cos\frac{\alpha_k-\beta_k}2.\end{align*}
The Lemma is therefore proved.
\end{proof}
The next lemma is an estimate of the difference quotient analogue.
We can no longer make use of estimates of partial derivatives
because we cannot define differentiation on the metric space
corresponding to $(k,1)$-generalized analysis for the failure of the existence of continuous rectifiable curves between two distinct points (see Remark 3.2. ii).
\begin{lem}\label{Lambda}
For $0<t<1$, $y\neq y_0$,
$$\left|\frac{\Lambda_{k,1}\left(x,y;t\right)-\Lambda_{k,1}\left(x,y_0;t\right)}{d\left(y,y_0\right)}\right|\leq\frac C{t^{2\langle k\rangle+N-\frac12}}\left(\tau_{y_0}\left(e^{-\frac ct\left\|\cdot\right\|}\right)(x)+\tau_y\left(e^{-\frac ct\left\|\cdot\right\|}\right)(x)\right).$$
\end{lem}

\begin{proof}
From \eqref{Lambda0}, we write
\begin{align*}&\;\;\;\;\;\frac{\Lambda_{k,1}\left(x,y;t\right)-\Lambda_{k,1}\left(x,y_0;t\right)}{d\left(y,y_0\right)}\\&=\frac1{(\sinh\;t)^{2\langle k\rangle+N-1}}\Bigg(\frac{e^{-\tanh\frac t2\left(\left\|x\right\|+\left\|y\right\|\right)}\tau_x\left(e^{-\frac1{\sinh\;t}\left\|\cdot\right\|}\right)(y)-e^{-\tanh\frac t2\left(\left\|x\right\|+\left\|y\right\|\right)}\tau_x\left(e^{-\frac1{\sinh\;t}\left\|\cdot\right\|}\right)(y_0)}{d\left(y,y_0\right)}\\&\;\;\;\;+\frac{e^{-\tanh\frac t2\left(\left\|x\right\|+\left\|y\right\|\right)}\tau_x\left(e^{-\frac1{\sinh\;t}\left\|\cdot\right\|}\right)(y_0)-e^{-\tanh\frac t2\left(\left\|x\right\|+\left\|y_0\right\|\right)}\tau_x\left(e^{-\frac1{\sinh\;t}\left\|\cdot\right\|}\right)(y_0)}{d\left(y,y_0\right)}\Bigg)\\&=\frac1{(\sinh\;t)^{2\langle k\rangle+N-1}}\left(I_1+I_2\right).\end{align*}

Notice that  $\sinh\;t$ behaves like $t$ for $0<t\leq 1$. For the
second part $I_2$. If $\left\|y\right\|=\left\|y_0\right\|$, then
$I_2=0$. If $\left\|y\right\|\neq\left\|y_0\right\|$, then from
the inequality
$$\left|\frac{e^{-\tanh\frac t2\cdot x_1^2}-e^{-\tanh\frac t2\cdot x_2^2}}{x_1-x_2}\right|\leq\underset x{max}\left|2\tanh\frac t2\cdot xe^{-\tanh\frac t2\cdot x^2}\right|\leq C_2\sqrt t,$$
we have
\begin{align*}\left|I_2\right|&=\left|\tau_x\left(e^{-\frac1{\sinh\;t}\left\|\cdot\right\|}\right)(y_0)\frac{e^{-\tanh\frac t2\left(\left\|x\right\|+\left\|y\right\|\right)}-e^{-\tanh\frac t2\left(\left\|x\right\|+\left\|y_0\right\|\right)}}{d\left(y,y_0\right)}\right|\\&\leq\tau_x\left(e^{-\frac1{\sinh\;t}\left\|\cdot\right\|}\right)(y_0)\cdot\left|\frac{e^{-\tanh\frac t2\left\|y\right\|}-e^{-\tanh\frac t2\left\|y_0\right\|}}{\sqrt{\left\|y\right\|}-\sqrt{\left\|y_0\right\|}}\right|\\&\leq C_2\sqrt t\tau_x\left(e^{-\frac{b_2}t\left\|\cdot\right\|}\right)(y_0).\end{align*}

For the first part $I_1$, from the inequality
\begin{align*}\left|\frac{e^{-\frac1{\sinh\;t}\cdot x_1^2}-e^{-\frac1{\sinh\;t}\cdot x_2^2}}{x_1-x_2}\right|&\leq\underset{x_2\leq x\leq x_1}{max}\left|\frac2{\sinh\;t}\cdot xe^{-\frac1{\sinh\;t}\cdot x^2}\right|\\&\leq\frac2{\sqrt{\sinh\;t}}e^{-\frac1{2\sinh\;t}\cdot x_2^2}\underset{x_2\leq x\leq x_1}{max}\left|\frac1{\sqrt{\sinh\;t}}\cdot xe^{-\frac1{2\sinh\;t}\cdot x^2}\right|\\&\leq C_1\frac1{\sqrt t}e^{-\frac{b_1}{t}\cdot x_2^2},\;\;\;x_1>x_2,\end{align*}
along with Lemma \ref{qusitrian} and \eqref{radial},
\begin{align*}\left|I_1\right|&\leq C_1e^{-\tanh\frac t2\left(\left\|x\right\|+\left\|y\right\|\right)}V_k\Bigg(\int_{-1}^1\left|\frac{e^{-\frac1{\sinh\;t}\left(\left\|x\right\|+\left\|y\right\|-\sqrt{2\left(\left\|x\right\|\left\|y\right\|+\left\langle\cdot,y\right\rangle\right)}u\right)}-e^{-\frac1{\sinh\;t}\left(\left\|x\right\|+\left\|y_0\right\|-\sqrt{2\left(\left\|x\right\|\left\|y_0\right\|+\left\langle\cdot,y_0\right\rangle\right)}u\right)}}{d\left(y,y_0\right)}\right|\\&\;\;\;\;\cdot\left(1-u^2\right)^{\frac N2+\left\langle k\right\rangle-2}du\Bigg)\left(x\right)\\&\leq\frac{C_1}{\sqrt t}e^{-\tanh\frac t2\left(\left\|x\right\|+\left\|y\right\|\right)}\int_{\mathbb{R}^N}\Bigg(\int_{\left\{u\in\lbrack-1,1\rbrack:\left\|y\right\|-\sqrt{2\left(\left\|x\right\|\left\|y\right\|+\left\langle\eta,y\right\rangle\right)}u>\left\|y_0\right\|-\sqrt{2\left(\left\|x\right\|\left\|y_0\right\|+\left\langle\eta,y_0\right\rangle\right)}u\right\}}\\&\;\;\;\;\; e^{-\frac1{2\sinh\;t}\left(\left\|x\right\|+\left\|y_0\right\|-\sqrt{2\left(\left\|x\right\|\left\|y_0\right\|+\left\langle\eta,y_0\right\rangle\right)}u\right)}\left(1-u^2\right)^{\frac N2+\left\langle k\right\rangle-2}du\\&\;\;\;\;+\int_{\left\{u\in\lbrack-1,1\rbrack:\left\|y\right\|-\sqrt{2\left(\left\|x\right\|\left\|y\right\|+\left\langle\eta,y\right\rangle\right)}u<\left\|y_0\right\|-\sqrt{2\left(\left\|x\right\|\left\|y_0\right\|+\left\langle\eta,y_0\right\rangle\right)}u\right\}}e^{-\frac1{2\sinh\;t}\left(\left\|x\right\|+\left\|y\right\|-\sqrt{2\left(\left\|x\right\|\left\|y\right\|+\left\langle\eta,y\right\rangle\right)}u\right)}\\&\;\;\;\;\cdot\left(1-u^2\right)^{\frac N2+\left\langle k\right\rangle-2}du\Bigg)d\mu_x(\eta)\\&\leq\frac{C_1}{\sqrt t}\left(\tau_{y_0}\left(e^{-\frac{b_1}t\left\|\cdot\right\|}\right)(x)+\tau_y\left(e^{-\frac{b_1}t\left\|\cdot\right\|}\right)(x)\right).\end{align*}

Thus
\begin{align*}\left|\frac{\Lambda_{k,1}\left(x,y;t\right)-\Lambda_{k,1}\left(x,y_0;t\right)}{d\left(y,y_0\right)}\right|&\leq\frac1{t^{2\langle k\rangle+N-1}}\left(C_2\sqrt t+\frac{C_1}{\sqrt t}\right)\left(\tau_{y_0}\left(e^{-\frac ct\left\|\cdot\right\|}\right)(x)+\tau_y\left(e^{-\frac ct\left\|\cdot\right\|}\right)(x)\right)\\&\leq\frac C{t^{2\langle k\rangle+N-\frac12}}\left(\tau_{y_0}\left(e^{-\frac ct\left\|\cdot\right\|}\right)(x)+\tau_y\left(e^{-\frac ct\left\|\cdot\right\|}\right)(x)\right).\tag*{\qedhere}\end{align*}
\end{proof}

\noindent{\it Proof of Theorem \ref{main}}. We only need to show
that the operator $\left(-\triangle_{k,1}\right)^{-i\sigma}$ is
$L^p$-bounded for $1<p\leq2$ and weakly $L^1$-bounded since it is
symmetric on
$L^2\left(\mathbb{R}^N,\vartheta_{k,1}\left(x\right)dx\right)$ and
its $L^p$-boundedness  for $2<p<\infty$ can be derived from the
duality argument. From \eqref{Kj}, we write
\begin{align*}K\left(x,y\right)&=\int_0^1\Lambda_{k,1}\left(x,y;t\right)t^{i\sigma-1}dt+\int_1^\infty\Lambda_{k,1}\left(x,y;t\right)t^{i\sigma-1}dt\\&=K^{(1)}\left(x,y\right)+K^{(2)}\left(x,y\right),\end{align*}
where $x,y\in \mathbb R^N$, $y\not\in G.x$. We claim that
$K\left(x,y\right)$ satisfies the condition in Theorem \ref{Hor}.

For the second part $K^{(2)}\left(x,y\right)$, by \eqref{Lambda1},
\eqref{prop1} and \eqref{inttauye},
\begin{align*}\int_{\mathbb{R}^N}\left|K^{(2)}\left(x,y\right)\right|\vartheta_{k,1}\left(x\right)dx&\leq C\int_{\mathbb{R}^N}\int_1^\infty e^{-\left({2\langle k\rangle+N-1}\right)t}\tau_y\left(e^{-b\left\|\cdot\right\|}\right)(x)\frac1t\vartheta_{k,1}\left(x\right)dtdx\\&= C\int_1^\infty\int_{\mathbb{R}^N}e^{-\left({2\langle k\rangle+N-1}\right)t}e^{-b\left\|x\right\|}\frac1t\vartheta_{k,1}\left(x\right)dx\\&\leq C\int_1^\infty e^{-\left({2\langle k\rangle+N-1}\right)t}\frac1tdt\leq C.\end{align*}
Then we have
$$\int_{d_G(x,y)>2d(y,y_0)}\left|K^{(2)}\left(x,y\right)-K^{(2)}(x,\;y_0)\right|\vartheta_{k,1}\left(x\right)dx\leq2\int_{\mathbb{R}^N}\left|K^{(2)}\left(x,y\right)\right|\vartheta_{k,1}\left(x\right)dx\leq C.$$

For the first part $K^{(1)}\left(x,y\right)$, from Lemma
\ref{Lambda},
\begin{align*}\left|K^{(1)}\left(x,y\right)-K^{(1)}\left(x,y_0\right)\right|&\leq\int_0^1\left|\Lambda_{k,1}\left(x,y;t\right)-\Lambda_{k,1}\left(x,y_0;t\right)\right|\frac1tdt\\&\leq Cd\left(y,y_0\right)\int_0^1\frac1{t^{2\langle k\rangle+N+\frac12}}\left(\tau_{y_0}\left(e^{-\frac ct\left\|\cdot\right\|}\right)(x)+\tau_y\left(e^{-\frac ct\left\|\cdot\right\|}\right)(x)\right)dt.\end{align*}
When $d_G(x,y)>2d(y,y_0)$, we have
$$d_G(x,y_0)\geq d_G(x,y)-d(y_0,y)>d(y,y_0),\;\;\;d_G(x,y)>d(y,y_0).$$
Then from \eqref{min}, for any $u\in [-1,1]$ and $\eta\in
co(G.x)$, we have
$$\sqrt{\left\|x\right\|+\left\|y_0\right\|-\sqrt{2\left(\left\|x\right\|\left\|y_0\right\|+\left\langle\eta,y_0\right\rangle\right)}u}\geq d_G(x,y_0)>d(y,y_0),$$
$$\sqrt{\left\|x\right\|+\left\|y\right\|-\sqrt{2\left(\left\|x\right\|\left\|y\right\|+\left\langle\eta,y\right\rangle\right)}u}\geq d_G(x,y)>d(y,y_0).$$
So
$$\tau_x\left(e^{-\frac ct\left\|\cdot\right\|}\right)(y_0)\leq\tau_x\left(e^{-\frac c{4t}\left(\sqrt{\left\|\cdot\right\|}+d(y,y_0)\right)^2}\right)(y_0),\;\;\;\tau_x\left(e^{-\frac ct\left\|\cdot\right\|}\right)(y)\leq\tau_x\left(e^{-\frac c{4t}\left(\sqrt{\left\|\cdot\right\|}+d(y,y_0)\right)^2}\right)(y).$$

Therefore, from \eqref{prop1} and \eqref{inttauye},
\begin{align*}&\int_{d_G(x,y)>2d(y,y_0)}\left|K^{(1)}\left(x,y\right)-K^{(1)}(x,\;y_0)\right|\vartheta_{k,1}\left(x\right)dx\\&\leq Cd(y,y_0)\int_0^1\frac1{t^{{2\langle k\rangle+N+\frac12}}}\bigg(\int_{\mathbb{R}^N}\tau_{y_0}\left(e^{-\frac c{4t}\left(\sqrt{\left\|\cdot\right\|}+d(y,y_0)\right)^2}\right)(x)\vartheta_{k,1}\left(x\right)dx\\&\;\;\;\;+\int_{\mathbb{R}^N}\tau_y\left(e^{-\frac c{4t}\left(\sqrt{\left\|\cdot\right\|}+d(y,y_0)\right)^2}\bigg)(x)\vartheta_{k,1}\left(x\right)dx\right)dt\\&=Cd(y,y_0)\int_0^1\frac1{t^{2\langle k\rangle+N+\frac12}}dt\int_{\mathbb{R}^N}2e^{-\frac c{4t}\left(\sqrt{\left\|x\right\|}+d(y,y_0)\right)^2}\vartheta_{k,1}\left(x\right)dx\\&\leq Cd(y,y_0)\int_0^\infty r^{2\left\langle k\right\rangle+N-2}dr\int_0^1\frac2{t^{2\langle k\rangle+N+\frac12}}e^{-\frac c{4t}\left(\sqrt r+d(y,y_0)\right)^2}dt\\&\leq Cd(y,y_0)\int_0^\infty\frac{r^{2\left\langle k\right\rangle+N-2}}{\left(\sqrt r+d(y,y_0)\right)^{2\left(2\langle k\rangle+N-\frac12\right)}}dr\int_0^\infty\frac2{u^{2\langle k\rangle+N+\frac12}}e^{-\frac c{4u}}du\\&\leq Cd(y,y_0)\int_0^\infty\frac1{\left(\sqrt r+d(y,y_0)\right)^3}dr=C.\end{align*}
The proof of Theorem \ref{main} is complete. $\hfill\Box$

\section*{Acknowledgments}The author would like to thank his adviser Nobukazu Shimeno for valuable comments and advice.
All data included in this study are available upon request by contact with the corresponding author.

\end{document}